\documentclass[preprint, 12pt]{elsarticle}
\usepackage[pdftex]{hyperref}
\usepackage{amsmath,amsthm,amsopn,amssymb,a4wide,times,varioref}
\usepackage[abbrev]{amsrefs}
\usepackage{enumitem}
\usepackage[english]{babel}
\usepackage[latin1]{inputenc}
\usepackage{amsfonts}
\usepackage[mathscr]{eucal}
\usepackage[utopia]{mathdesign}
\usepackage{tikz}
\usepackage{fancyhdr}
\numberwithin{equation}{section}
\setlist{label={$($\arabic{enumi}\kern1pt$)$}}
\setlist[itemize,1]{label={\textbullet}}
\newtheorem{fed}{Definition}[section]
\newtheorem{theorem}[fed]{Theorem}
\newtheorem{proposition}[fed]{Proposition}
\newtheorem{corollary}[fed]{Corollary}
\newtheorem{lemma}[fed]{Lemma}

\theoremstyle{remark}
\newtheorem{remark}[fed]{Remark}
\newtheorem{finalremark}[fed]{Final remark}

\newtheorem{example}[fed]{Example}

\theoremstyle{definition}
\newtheorem{definition}[fed]{Definition}

\newcommand{\ip}[2]{{
    \left<
      #1,#2
    \right>}}

\newcommand{\roverline}[1]{\mathpalette\doroverline{#1}}
\newcommand{\doroverline}[2]{\overline{#1#2}}

\def\H{\mathcal{H}}
\def\K{\mathcal{K}}
\def\M{\mathcal{M}}
\def\cl{\roverline}

\def\D{{\mathcal{D}}}

\def\M{{\mathcal{M}}}
\def\N{{\mathcal{N}}}

\def\P{{\mathcal{P}}}
\def\Q{{\mathcal{Q}}}

\def\K{{\mathcal{K}}}
\def\L{{\mathcal{L}}}

\newcommand{\rs}{{\phantom{1}}}
\newcommand{\hs}{{\hphantom{*}}}

\newcommand\encircle[1]{%
  \tikz[baseline=(X.base)] 
    \node (X) [draw, shape=circle, inner sep=0] {#1};}
\newcommand{\ssharp}{{\scriptscriptstyle\sharp}}
\newcommand{\core}{{\encircle{$\scriptscriptstyle\sharp$}}}

\newcommand{\minus}{%
  \mathrel{\vbox{\offinterlineskip\ialign{%
    \hfil##\hfil\cr
    $\scriptscriptstyle-$\cr
    \noalign{\kern-.1ex}
    $\leq$\cr
}}}}
\newcommand{\notminus}{%
  \mathrel{\vbox{\offinterlineskip\ialign{%
    \hfil##\hfil\cr
    $\scriptscriptstyle-$\cr
    \noalign{\kern-.1ex}
    $\not\leq$\cr
}}}}
\newcommand{\stella}{%
  \mathrel{\vbox{\offinterlineskip\ialign{%
    \hfil##\hfil\cr
    $\scriptstyle*$\cr
    $\leq$\cr
}}}}
\newcommand{\lstella}{%
  \mathrel{\vbox{\offinterlineskip\ialign{%
    \hfil##\hfil\cr
    $\scriptstyle*$\cr
    \noalign{\kern.45ex}}}\kern-.3ex\leq}}
\newcommand{\rstella}{%
  \mathrel{\leq\kern-.4ex\vbox{\offinterlineskip\ialign{%
    \hfil##\hfil\cr
    $\scriptstyle*$\cr
    \noalign{\kern.45ex}}}}}
\newcommand{\sharppo}{%
  \mathrel{\vbox{\offinterlineskip\ialign{%
    \hfil##\hfil\cr
    $\scriptscriptstyle\sharp$\cr
    \noalign{\kern-.05ex}
    $\leq$\cr
}}}}
\newcommand{\corepo}{%
  \mathrel{\vbox{\offinterlineskip\ialign{%
    \hfil##\hfil\cr
    ${\encircle{$\scriptscriptstyle\sharp$}}$\cr
    \noalign{\kern-.05ex}
    $\leq$\cr
}}}}

\def\dotplus{\mathrel{\dot{+}}}

\def\Oplus{\mathrel{\oplus}}

\newcommand{\lminus}{%
  \mathrel{\vbox{\offinterlineskip\ialign{%
    \hfil##\hfil\cr
    $\scriptscriptstyle-$\cr
    \noalign{\kern.1ex}}}\kern-.45ex\leq}}
\newcommand{\rminus}{%
  \mathrel{\leq\kern-.45ex\vbox{\offinterlineskip\ialign{%
    \hfil##\hfil\cr
    $\scriptscriptstyle-$\cr
    \noalign{\kern.1ex}}}}}

\date{}
\begin{document}

\begin{frontmatter}

\title{The minus order and range additivity}

\author[author1]{Marko S. Djiki\'{c}\fnref{fn1}}
\ead{marko.djikic@gmail.com}

\author[author2]{Guillermina Fongi\fnref{fn2}}
\ead{gfongi@conicet.gov.ar}

\author[author2,author3]{Alejandra Maestripieri \corref{cor1}\fnref{fn3}}
\ead{amaestri@fi.uba.ar}

\cortext[cor1]{Corresponding author}
\fntext[fn1]{Supported by Grant No. 174007 of the Ministry of Education, Science and Technological Development of the Republic of Serbia}
\fntext[fn2]{Partially supported by   CONICET (PIP 426 2013-2015)}
\fntext[fn3]{Partially supported by CONICET (PIP 168 2014-2016)}

\address[author1]{Department of Mathematics, Faculty of Sciences and Mathematics, University of Ni\v{s}, 
Vi\v{s}egradska 33, 18000 Ni\v{s}, Serbia.}

\address[author2]{Instituto Argentino de Matem\'atica ``Alberto P. Calder\'on"
 Saavedra 15, Piso 3 (1083), Buenos Aires, Argentina.}

\address[author3]{Departamento de Matem\'atica, Facultad de Ingenier\'ia, Universidad de Buenos Aires.}

\begin{abstract}
  We study the minus order on the algebra of bounded linear operators
  on a Hilbert space.  By giving a characterization in terms of range
  additivity, we show that the intrinsic nature of the minus order is
  algebraic.  Applications to generalized inverses of the sum of two
  operators, to systems of operator equations and to optimization
  problems are also presented.
\end{abstract}


\end{frontmatter}

\section{Introduction}

The minus order was introduced by Hartwig \cite{Har80} and
independently by Nambooripad \cite{Nam80}, in both cases on
semigroups, with the idea of generalizing some classical partial
orders.  It was extended to operators in infinite dimensional spaces
independently by Antezana, Corach and Stojanoff \cite{AntCorSto06}
and by $\check{\textrm{S}}$emrl \cite{Sem10}.  There is now an
extensive literature devoted to this order and other related partial
orders on matrices, operators and elements of various algebraic
structures.  See for example, \cite{Mit91,Mit86,BakTre10}.
 
The main goal of this work is to obtain a new characterization of the
minus order for operators acting on Hilbert spaces in terms of the so
called range additivity property.  Given two linear bounded operators
$A$ and $B$ acting on a Hilbert space $\H$, we say that $A$ and $B$
have the \emph{range additivity property} if $R(A+B)=R(A)+R(B)$, where
$R(T)$ stands for the range of an operator $T$.  Operators with this
property have been studied in \cite{AriCorGon13} and
\cite{AriCorMae15} (see also \cite{CorFonMae16}).  Recall that if $A$
and $B$ are two bounded linear Hilbert space operators then $A\minus
B$ (where the symbol ``$\minus$'' stands for the minus order of
operators) if and only if there are oblique projections $P$ and $Q$
such that $A = PB$ and $A^* = QB^*$.  In this paper, we prove that
this is equivalent to the range of $B$ being the direct sum of ranges
of $A$ and $B-A$ and the range of $B^*$ being the direct sum of ranges
of $A^*$ and $B^*-A^*$.  Thus the minus order is intrinsically
algebraic in nature.  This plays an equivalent role to a known
characterization when $A$ and $B$ are matrices \cite{Mit91,Har80};
that $A\minus B$ if and only if the rank of $B-A$ is the difference of
the rank of $B$ and the rank of $A$.

As a consequence, diverse concepts that have been developed for
matrices and operators are in fact manifestations of the minus order.
These include, for example, the notions of weakly bicomplementary
matrices defined due to Werner \cite{Wer86}, and quasidirect addition
of operators defined by Le\v snjak and \v Semrl \cite{LesSem96}.
Although in these papers the minus order does not appear explicitly,
these notions when applied to operators $A$ and $B$ are equivalent to
saying that $A\minus A+B$.  The minus order also lurks in the papers
of Baksalary and Trenkler \cite{BakTre11}, Baksalary,
$\check{\textrm{S}}$emrl and Styan \cite{BakSemSty96}, Mitra
\cite{Mit72} and Arias, Corach and Maestripieri \cite{AriCorMae15}.

The minus order can be weakened to what we call left and right minus
orders.  As with the minus order, these orders are easily derived from
a range additivity condition.  It happens that they truly differ from
the minus order only in the infinite dimensional setting.  When
$A\minus B$, we give some applications to formulas for generalized
inverses of sums $A+B$ in terms of generalized inverses of $A$ and
$B$, and we show that certain optimization problems involving the
operator $A+B$ can be decoupled into a system of similar problems for
$A$ and $B$.

The paper is organized as follows.  In Section 2 we collect some
useful known results about range additivity, while in Section 3, the
minus order is defined and the connection with range additivity is
made.  Motivated by the concepts of the left and the right star
orders, we define left and the right minus orders on $ L(\H)$.  For
matrices, these are equivalent to the minus order, with differences
only emerge in the infinite dimensional context.
Proposition~\ref{leftminus and projections} characterizes the left
minus order in terms of densely defined, though not necessarily
bounded, projections.  Additionally, the left minus, the right minus
and the minus orders are characterized in terms of (densely defined)
inner generalized inverses, generalizing a matricial result (see
\cite{Mit86}).

Finally, Section 4 is devoted to applications.  We begin by relating
the minus partial order to some formulas for reflexive inner inverses
of the sum of two operators.  In particular, we give an alternative
proof for the Fill-Fishkind formula for the Moore-Penrose inverse of a
sum, as found in \cite{FilFis99} for matrices and extended to $L(\H)$
by Arias et al. \cite{AriCorMae15}.  We also apply the new
characterization of the minus order to systems of equations and least
squares problems.  We include a final remark about a possible
generalization of the minus order involving densely defined
projections with closed range.

\section{Preliminaries}

Throughout, $(\H, \ip{\cdot}{\cdot})$ denotes a complex Hilbert space
and $L(\H)$ the algebra of linear bounded operators on $\H$, $\mathcal
Q$ is the subset of $L(\H)$ of oblique projections, i.e., $\mathcal
Q=\{Q\in L(\H):\, Q^2=Q\}$ and $\mathcal P$ the subset of $\Q$ of
orthogonal projections, i.e., $\mathcal P=\{P\in L(\H): P^2=P=P^*\}$.

Given $\mathcal{M}$ and $\mathcal{N}$ two closed subspaces of $\H$,
write $\mathcal{M}\dotplus \mathcal{N}$ for the direct sum of
$\mathcal M$ and $\mathcal N$, $\mathcal{M}\Oplus \mathcal{N}$ the
orthogonal sum and $\mathcal M\ominus \mathcal N=\mathcal M \cap
(\mathcal M\cap\mathcal N)^\perp$.  If $\mathcal{M}\dotplus
\mathcal{N}=\H$, the oblique projection with range $\mathcal{M}$ and
null space $\mathcal{N}$ is $P_{\mathcal{M}// \mathcal{N}}$ and
$P_{\mathcal{M}}=P_{\mathcal{M}//\mathcal{M}^\perp}$ is the orthogonal
projection onto $\M$.

For $A \in L(\H)$, $R(A)$ stands for the range of $A$, $N(A)$ for its
null space and $P_A$ for $P_{\cl{R(A)}}$.  The Moore-Penrose inverse of
$A$ is the (densely defined) operator $A^\dagger: R(A)\Oplus
R(A)^\bot\rightarrow \H$, defined by
$A^\dagger|_{R(A)}=(A|_{N(A)^\perp})^{-1}$ and
$N(A^\dagger)=R(A)^\bot$.  It holds that $A^\dagger\in L(\H)$ if and
only if $A$ has a closed range.

\medskip

Given $\mathcal{M}$ and $\mathcal{N}$ two closed subspaces of $\H$,
the \textit{minimal angle} between $\mathcal M$ and $\mathcal N$ is
$\alpha_0(\mathcal M,\mathcal N) \in [0,\pi/2]$, the cosine of which
is
\begin{equation*}
  c_0(\mathcal M,\mathcal N)=\sup\,\left\{|\ip{\xi}{\eta} |:~ \xi
    \in\mathcal M,~\| \xi 
    \|\leq 1,~~\eta\in\mathcal N,~\|\eta\|\leq 1\right\} \,\in[0,1].
\end{equation*}

When the minimal angle between $\M$ and $\N$ is strictly less that 1,
then the sum $\M+\N$ is closed and direct, moreover, we have the
following.

\begin{proposition} \label{angle c0} Let $\mathcal{M}$ and
  $\mathcal{N}$ be two closed subspaces of $\H$.  The following
  statements are equivalent:
  \begin{enumerate}
  \item $c_0(\mathcal M,\mathcal N)<1$;
  \item $\mathcal M\dotplus \mathcal N$ is closed;
  \item  $\H=\mathcal M^\perp+\mathcal N^\perp$.
  \end{enumerate}
\end{proposition}

For a proof, see Lemma 2.11 and Theorem 2.12 in \cite{Deu95}.

\medskip

For $ A, B\in L(\H)$, it always holds that $R(A+B)\subseteq
R(A)+R(B)$.  We say that $ A$ and $ B$ have the \textit{range
  additivity property} if $R(A +B) =R(A) +R(B)$.  In this case,
$R(A)\subseteq R(A+B)$.  Conversely, if $R(A)\subseteq R(A+B)$ then,
for $x\in \H$, $Bx=(A+B)x-Ax\in R(A+B)$.  We have proved the
following.

\begin{lemma}[{\cite[Proposition~2.4]{AriCorMae15}}]
  \label{char1 RAP} 
  For $A, B\in L(\H)$, $R(A+B)=R(A)+R(B)$ if and only if
  $R(A)\subseteq R(A+B)$.
\end{lemma}

Operators having the range additivity property were characterized in
\cite[Theorem~2.10]{AriCorMae15}.  Closely related is the following
for operators $A, B\in L(\H)$ satisfying the condition $R(A)\cap
R(B)=\{0\}$.

\begin{proposition}[{\cite[Theorem~2.10]{AriCorMae15}}]
  \label{char2 RAP} 
  Consider $A, B\in L(\H)$ such that $R(A)\cap R(B)=\{0\}$ then
  $R(A+B)=R(A)\dotplus R(B)$ if and only if $ \H=N(A)+N(B)$.
\end{proposition}

\medskip

The next result will be useful in characterizing the minus order in
Section~\ref{sec:minus-order} (see
\cite[Proposition~2.2]{AriCorMae15}).

\begin{proposition}\label{equiv range sum, kernels}
  For $A, B\in L(\H)$ consider the following statements:
  \begin{enumerate}
  \item $\roverline{R(A^*)} \dotplus  \roverline{R(B^*)}$ is closed;
  \item there exists  $Q\in \Q$ such that $A^*=Q(A^*+B^*)$;
  \item $N(A)+N(B)=\H$;
  \item $R(A+B)=R(A)+R(B)$.
  \end{enumerate}
  Then $(1)\Leftrightarrow (2) \Leftrightarrow (3) \Rightarrow (4)$.
  The implication $(4) \Rightarrow (3)$ holds if $R(A)\cap
  R(B)=\{0\}.$
\end{proposition}

For a proof of $(1)\Leftrightarrow (2)$ see
\cite[Proposition~4.13]{AntCorSto06}.  $(1)\Leftrightarrow (3)$ was
stated in Proposition~\ref{angle c0}.  The implication $(3)
\Rightarrow (4)$ follows from the proof of
\cite[Proposition~2.8]{AriCorGon15}.  $(4) \Rightarrow (3)$ follows
from Proposition~\ref{char2 RAP} since $R(A)\cap R(B)=\{0\}$.

\section{The minus order}
\label{sec:minus-order}

Different definitions have been given for the \emph{minus (partial)
  order}.  For operators we offer one which equivalent to those
appearing in \cite{AntCorSto06} and \cite{Sem10}.

\begin{definition}
  For $A, B\in L(\H)$, $A \minus B$ if there exist $P, Q \in \Q$ such
  that $A=PB$ and $A^*=QB^*$.
\end{definition}

Proofs that $\minus$ is a partial order on $L(\H)$ can be found in
\cite[Corollary~4.14]{AntCorSto06} and \cite[Corollary~3]{Sem10}.
It is easy to see that the ranges of $P$ and $Q$ can be fixed so that
$R(P)=\roverline{R(A)}$ and $R(Q)=\roverline{R(A^*)}$.  For details,
see \cite[Proposition~4.13]{AntCorSto06} and the definition of minus
order in \cite{Sem10}.

\medskip

In the next proposition we collect some characterizations of the minus
order in terms of angle conditions and sum of closed subspaces.

\begin{proposition}\label{equivalencias minus order}
  Consider $A,B\in L(\H)$.  The following statements are equivalent:
  \begin{enumerate}
  \item $A\minus B$;
  \item $c_0(\roverline{R(A)}, \roverline{R(B-A)}) <1$ and
    $c_0(\roverline{R(A^*)}, \roverline{R(B^*-A^*)}) < 1$;
  \item $\roverline {R(B)}=\roverline{R(A)} \dotplus
    \roverline{R(B-A)}$ and $\roverline {R(B^*)}=\roverline{R(A^*)}
    \dotplus \roverline{R(B^*-A^*)}$;
  \item $ N(A)+ N(B-A)=N(A^*)+N(B^*-A^*)=\H$;
  \item there exists $P\in \Q$ such that $A=PB$ and $R(A)\subseteq
    R(B)$.
  \end{enumerate}
\end{proposition}

\begin{proof}
  The equivalences $(1) \Leftrightarrow (2) \Leftrightarrow (4)$
  follow applying the definition of the minus order and
  Proposition~\ref{angle c0} to the operators $A$, $B-A$, $A^*$ and
  $B^*-A^*$, see also \cite[Proposition~4.13]{AntCorSto06}.

  For $(2) \Leftrightarrow (3)$, suppose that $c_0(\roverline{R(A)},
  \roverline{R(B-A)}) <1$ and $c_0(\roverline{R(A^*)},
  \roverline{R(B^*-A^*)}) <1$.  Then $\cl{R(A)}\dotplus \cl{R(B-A)}$
  and $\cl{R(A^*)}\dotplus \cl{R(B^*-A^*)} $ are closed.  In this
  case, $\cl{R(B)}\subseteq \cl{R(A)}\dotplus \cl{R(B-A)}$.  On the
  other hand, applying Proposition~\ref{equiv range sum, kernels},
  there exists $Q\in \Q$ such that $A^*=QB^*$.  Then $N(B^*) \subseteq
  N(A^*)$ and $N(B^*)\subseteq N(B^*-A^*)$, or $\cl{R(A)}\subseteq
  \cl{R(B)}$ and $\cl{R(B-A)}\subseteq \cl{R(B)}$.  Then $\roverline
  {R(B)}=\roverline{R(A)} \dotplus \roverline{R(B-A)}$.  Similarly,
  $\roverline {R(B^*)}=\roverline{R(A^*)} \dotplus
  \roverline{R(B^*-A^*)}$.  See also \cite[Theorem~2]{Sem10}.
  Conversely, if item 3 holds, then $\cl{R(A)}\dotplus \cl{R(B-A)} $
  and $\cl{R(A^*)}\dotplus \cl{R(B^*-A^*)} $ are closed or
  equivalently, by Proposition~\ref{angle c0}, item 2 holds.

  Next consider $(1) \Leftrightarrow (5)$.  If $A\minus B$ then
  $A=PB=BQ^*$ with $P,Q\in \Q$, so that $A=PB$ and $R(A)\subseteq
  R(B)$.  Conversely, suppose $R(A)\subseteq R(B)$ and there exists
  $P\in \Q$ such that $A=PB$.  Then by Lemma~\ref{char1 RAP} it holds
  that $R(B)=R(A)+R(B-A)$.  Moreover $R(A)\cap R(B-A)=\{0\}$ because
  $R(A)\subseteq R(P)$ and $R(B-A)\subseteq N(P)$, so that
  $R(B)=R(A)\dotplus R(B-A)$.  In this case, $(4) \Rightarrow (2)$ of
  Proposition~\ref{equiv range sum, kernels} can be applied so that
  there exists $Q\in \Q$ such that $A^*=QB^*$.  Therefore $A\minus B$.
\end{proof}

\medskip

The following is a key result that will be useful on many occasions
throughout the paper.  It gives a new characterization of the minus
partial order in terms of the range additivity property, showing that
the minus order has an algebraic nature.

\begin{theorem}
  \label{equiv minus and RAP}
  Consider $A,B\in L(\H)$.  Then the following assertions are
  equivalent:
  \begin{enumerate}
  \item $A\minus B$;
  \item $R(B)=R(A)\dotplus R(B-A)$ and $R(B^*)=R(A^*)\dotplus
    R(B^*-A^*)$.
\end{enumerate}
\end{theorem}

\begin{proof}
  Suppose that $A\minus B$.  By Proposition~\ref{equivalencias minus
    order}, it follows that $\roverline{R(A)}\dotplus
  \roverline{R(B-A)} $ and $\roverline{R(A^*)} \dotplus
  \roverline{R((B-A)^*)}$ are closed.  In particular, $R(A)\cap
  R(B-A)=R(A^*)\cap R(B^*-A^*)=\{0\}$.  Also, it follows from
  Proposition~\ref{equiv range sum, kernels} that
  $R(B^*)=R(A^*)+R(B^*-A^*)$ and $R(B)=R(A)+R(B-A)$.  Therefore,
  $R(B)=R(A)\dotplus R(B-A)$ and $R(B^*)=R(A^*)\dotplus R(B^*-A^*)$.

  Conversely, suppose that $R(B)=R(A)\dotplus R(B-A)$ and
  $R(B^*)=R(A^*)\dotplus R(B^*-A^*)$.  Applying $(4)\Rightarrow (1)$
  in Proposition~\ref{equiv range sum, kernels}, it follows that
  $\roverline{R(A)}\dotplus \roverline{R(B-A)}$ and
  $\roverline{R(A^*)}\dotplus \roverline{R(B^*-A^*)}$ are closed.
  Hence, by Proposition~\ref{equivalencias minus order} and
  Proposition~\ref{angle c0}, $A\minus B$.
\end{proof}

\medskip

Let $A_i\in L(\H)$ for $1\leq i \leq k$.  Le\v snjak and \v Semrl
\cite{LesSem96} give the following definition: the operator
$A=\overset{k}{\underset{i=1}{\sum}} A_i$ is the \textit{quasidirect
  sum} if the range of $A$ is the direct sum of the ranges of the
$A_i$s and the closure of the range of $A$ is the direct sum of the
closures of the ranges of the $A_i$s.  The next result may be restated
as saying that $B$ is the quasidirect sum of $A$ and $B-A$ if and only
if $A\minus B$.

\begin{corollary}\label{ch3l2} 
  If $A,B\in L(\H)$, the following conditions are equivalent:
  \begin{enumerate}
  \item $A \minus B$;
  \item $R(B)=R(A)\dotplus R(B-A)$ and
    $\roverline{R(B)}=\roverline{R(A)}\dotplus \roverline{R(B-A)}.$
  \end{enumerate}
\end{corollary}

\begin{proof}
  $(1) \Rightarrow (2)$ follows from Proposition~\ref{equivalencias
    minus order} and Theorem~\ref{equiv minus and RAP}.

  For the converse, since $\roverline{R(A)}\dotplus
  \roverline{R(B-A)}$ is closed, from Proposition~\ref{equiv range
    sum, kernels} we have that $R(B^*)=R(A^*)+R(B^*-A^*)$.  To see
  that this sum is direct, applying Proposition~\ref{equiv range sum,
    kernels} again and using the fact that $R(B)=R(A)\dotplus R(B-A)$
  we get that $R(A^*)\cap R(B^*-A^*)=\{ 0 \}$.  Thus
  $R(B^*)=R(A^*)\dotplus R(B^*-A^*)$, and so by Theorem~\ref{equiv
    minus and RAP}, $A\minus B$.
\end{proof}

\medskip

The next result shows the behavior of the minus order when the
operators have closed ranges.

\begin{corollary}
  \label{A minus B and rg B cl then rg A, rg B-A cl}
  Consider $A, B\in L(\H)$ such that $A\minus B$.  Then $R(B)$ is
  closed if and only if $R(A)$ and $R(B-A)$ are closed.
\end{corollary}

\begin{proof}
  If $A\minus B$, then by Corollary~\ref{ch3l2},
  $\roverline{R(B)}=\roverline{R(A)}\dotplus \roverline{R(B-A)}$ and
  $R(B)={R(A)}\dotplus {R(B-A)}$.  If $R(B)$ is closed then
  $\roverline{R(A)}\dotplus \roverline{R(B-A)}={R(A)}\dotplus
  {R(B-A)}$.  Hence $\roverline{R(A)}=R(A)$ and
  $\roverline{R(B-A)}=R(B-A)$.  In fact, given $x\in
  \roverline{R(A)}$, then $x\in R(A)\dotplus R(B-A)$, so that there
  exist $x_1 \in R(A)$ and $x_2\in R(B-A)$ such that $x=x_1+x_2$.  But
  $x-x_1=x_2 \in \roverline{R(A)}\cap \roverline{R(B)}=\{0\}$, and so
  $x=x_1\in R(A)$; that is, $\roverline{R(A)}=R(A)$.  Similarly,
  $\roverline{R(B-A)}=R(B-A)$.  The converse follows by
  Corollary~\ref{ch3l2}.
\end{proof}

\medskip

\subsection{The left  and right minus orders}

In this section we define the \textit{left} and \textit{right minus
  orders} and show that they are a generalization of the left and
right star orders.  As we will see, these orders are really only
interesting on infinite dimensional spaces.  For matrices, they
coincide with the minus order.
 
We begin analyzing the properties of the left and right star orders.
Originally, Drazin~\cite{Dra78} introduced the star order on
semigroups with involutions, Baksalary and Mitra~\cite{BakMit91}
defined the left and right star orders for complex matrices, and
later, Antezana, Cano, Mosconi and Stojanoff~\cite{AntCanMosSto10}
extended the star order to the algebra of bounded operators on a
Hilbert space.  See also Dolinar and Marovt \cite{DolMar11}, Deng
and Wang \cite{DenWan12} and Djiki$\acute{\textrm c}$
\cite{Dji16}.

Given $A, B\in L(\H)$, the \textit{star order}, \textit{left star
  order} and \textit{right star order} are respectively defined by
\begin{itemize}
\item $A\stella B$ if and only if $A^*A=A^*B$ and $AA^*=BA^*$,
\item $A\lstella B$ if and only if $A^*A=A^*B$ and $R(A)\subseteq
  R(B)$, and
\item $A\rstella B$  if and only if $AA^*=BA^*$ and $R(A^*)\subseteq
  R(B^*)$.
\end{itemize}

\medskip

If $A, B\in L(\H)$, then $A\stella B$ if and only if there exist $P,
Q\in \P$ such that $A=PB$ and $A^*=QB^*$ (see
\cite[Proposition~2.3]{AntCanMosSto10} or
\cite[Theorem~5]{DolMar11}).  We can always take $P=P_A$ and
$Q=P_{A^*}$.

\medskip

The next result is a straightforward consequence of
\cite[Theorem~2.1]{DenWan12}.  We include a simple proof.

\begin{proposition}
  \label{left star order entonces minus order}
  Let $A, B\in L(\H)$.  If $A \lstella B$ then $A\minus B$.
\end{proposition}

\begin{proof}
  If $A \lstella B$ , then $A^*A=A^*B$, or equivalently $A^*(A-B)=0$.
  Hence $P_A(A-B)=0$, or $A=P_AB$.  Conversely, if $A = P_A B$, then
  $A^*A = A^*B$.  So $A\lstella B$ is equivalent to $A=P_AB$ and
  $R(A)\subseteq R(B)$.  By Proposition~\ref{equivalencias minus
    order}(5), this gives $A\minus B$.
\end{proof}

The following results characterize the left and right star orders in
terms of an orthogonal range additivity property.

\begin{proposition}
  \label{*< sii rangos ortog}
  For $A, B\in L(\H)$, $A \lstella B$ if and only if $R(B)=R(A)\Oplus
  R(B-A)$.
\end{proposition}

\begin{proof}
  From the proof of Proposition~\ref{left star order entonces minus
    order}, $A \lstella B$ if and only if $A=P_AB$ and $R(A)\subseteq
  R(B)$.  Thus $R(B)=R(A)\Oplus R(B-A)$ since $R(B-A)\subseteq
  N(P_A)=R(A)^\perp$.

  Conversely, if $R(B)=R(A)\Oplus R(B-A)$, then $R(A)\subseteq R(B)$
  and $R(B-A) \subseteq R(A)^\perp$, so that $A=P_A B$.  Hence $A
  \lstella B$.
\end{proof}

\begin{corollary}
  \label{<* sii rangos de adjuntos ortog}
  For $A, B\in L(\H)$, $A \rstella B$ if and only if
  $R(B^*)=R(A^*)\Oplus R(B^*-A^*)$.
\end{corollary}

The next characterization of the star order follows from the previous
results (or alternatively, from Theorem~\ref{equiv minus and RAP}).

\begin{corollary}
  Given $A, B\in L(\H)$, the following statements are equivalent:
  \begin{enumerate}
  \item $A\stella B$;
  \item $A\lstella B$ and $A\rstella B$;
  \item $ {R(B)}={R(A)} \Oplus {R(B-A)}$ and $ {R(B^*)}={R(A^*)}
    \Oplus{R(B^*-A^*)}$.
  \end{enumerate}
\end{corollary}

\begin{proof}
  Obviously, $A\lstella B$ and $A\rstella B$.  On the other hand, if
  $A\stella B$, then by the proof of Proposition~\ref{left star order
    entonces minus order}, $A = P_A B$ and $A^* = P_{A_\rs^*}B^*$.
  Hence $R(A^*) \subseteq R(B^*)$ and $R(A) \subseteq R(B)$.  Thus
  $(1) \Leftrightarrow (2)$.  The equivalence of these to $(3)$
  follows from Proposition~\ref{*< sii rangos ortog}.
\end{proof} 

\medskip

As a generalization of the left and right star orders, we now define
the \emph{left} and \emph{right minus orders}.

\begin{definition}
  For $A, B\in L(\H)$,
  \begin{itemize}
  \item $A \lminus B$ if and only if $R(B)=R(A)\dotplus R(B-A)$, and
  \item $A \rminus B$ if and only if $R(B^*)=R(A^*)\dotplus
    R(B^*-A^*)$.
  \end{itemize}
\end{definition}

\begin{proposition}
  The relations $\lminus$ and $\rminus$ define partial orders.
\end{proposition}

\begin{proof}
  We only give the proof for $\lminus$, since the proof for $\rminus$
  is identical.

  First of all, $\lminus$ is clearly reflexive.  So consider $A, B\in
  L(\H)$ such that $A\lminus B$ and $B \lminus A$.  Then
  $R(B)=R(A)\dotplus R(B-A)$ and $R(A)=R(B)\dotplus R(B-A)$.  From the
  last equality $R(B-A)\subseteq R(A)$.  But $R(B-A)\cap R(A)=\{0\}$,
  so that $R(B-A)=\{0\}$.  Therefore $A=B$ thus $\lminus$ is
  antisymmetric.

  To prove $\lminus$ is transitive, consider $A,B, C\in L(\H)$ such
  that $A \lminus B$ and $B\lminus C$.  Then $R(B)=R(A)\dotplus
  R(B-A)$ and $R(C)=R(B)\dotplus R(C-B)$.  Since $R(A)\subseteq
  R(B)\subseteq R(C)$, by Lemma~\ref{char1 RAP}, $R(C)=R(A)+R(C-A)$.
  It remains to show that $R(A)\cap R(C-A)=\{0\}$.  Since $R(A)\cap
  R(C-A)\subseteq R(A)\cap (R(C-B)+R(B-A))$ we can write $x\in R(A)$
  as $x=x_1+x_2, \, x_1\in R(C-B)$ and $x_2\in R(B-A)$.  Then
  $x-x_2=x_1\in R(B)\cap R(C-B)=\{0\}$, and so $x=x_2$.  Hence $x\in
  R(A)\cap R(B-A)=\{0\}$; that is $x=0$ and $R(A)\cap R(C-A)=\{0\}$.
  This implies that $R(C)=R(A)\dotplus R(C-A)$, or equivalently,
  $A\lminus C$, and so $\lminus$ is transitive.
\end{proof}

\medskip

The next corollary is a consequence of Theorem~\ref{equiv minus and
  RAP}.

\begin{corollary}
  For $A, B\in L(\H)$, $A\lminus B$ and $A \rminus B$ if and only if
  $A \minus B$.
\end{corollary}

It follows from Proposition~\ref{equiv range sum, kernels} that $A
\lminus B$ if and only if $A^*=QB^*$ for $Q\in \Q$ and $R(A)\cap
R(B-A)=\{0\}$.  There is also a characterization of the left minus
order similar to that of the left star order as found in the proof of
Proposition~\ref{left star order entonces minus order}.  We leave the
obvious version for the right minus order unstated.

\begin{proposition}\label{leftminus and projections}
  For $A, B\in L(\H)$, $A \lminus B$ if and only if there exists a
  (possibly unbounded) densely defined projection $P$ such that $A=PB$
  and $R(A)\subseteq R(B)$.
\end{proposition}

\begin{proof}
  If $A \lminus B$ then $ R(B)=R(A)\dotplus R(B-A)$ so that
  $R(A)\subseteq R(B)$.  Define $P=P_{R(A)//{R(B-A)\Oplus
      N(B_\rs^*)}}$.  Then $P$ is a densely defined projection and it
  is easy to check that $A=PB$.  Conversely, if $A=PB$ for a densely
  defined projection and $R(A)\subseteq R(B)$ then $R(B)=R(A)+R(B-A)$
  by Lemma~\ref{char1 RAP}, and the sum is direct since $R(A)\subseteq
  R(P)$ and $R(B-A)\subseteq N(P)$.
\end{proof}

\begin{remark}\label{star and minus}
  The minus order can be seen as a star order after applying suitable
  weights to the Hilbert spaces involved.  Recall that, if $A,B\in
  L(\H, \K)$ are such that $A\minus B$ then there exist projections
  $P\in L(\K)$ and $Q\in L(\H)$ such that $A=PB=BQ$.  The operators
  $W_1=Q^*Q+(I-Q^*)(I-Q)\in L(\H)$ and $W_2=P^*P+(I-P^*)(I-P)\in
  L(\K)$ are positive and invertible.  Hence the inner products in
  $\H$ and $\K$ respectively,
  \begin{equation*}
    \ip{ x}{ y}_{W_1}= \ip{W_1 x}{ y}, \textrm{ for } x, y \in \H \quad
    \textrm{ and } \quad \ip{ z}{ w }_{W_2}= \ip{W_2 z}{ w}, \textrm{
      for } z, w \in \K
  \end{equation*}
  give rise to equivalent norms.  With these new inner products, the
  projections $P$ and $Q$ are orthogonal in $\K_{W_2}=(\K,
  \ip{\cdot}{\cdot}_{W_2})$ and $\H_{W_1}=(\H,
  \ip{\cdot}{\cdot}_{W_1})$, respectively, and so $A\stella B$.

  On the other hand, $A\lminus B$ if and only if there exists a
  densely defined projection $P$ such that $A=PB$ and $R(A)\subseteq
  R(B)$.  In this case, it is possible to find a positive and
  invertible weight $W_2$ on $\K$ such that $P$ is symmetric with
  respect to $\ip{\cdot}{\cdot}_{W_2}$ (or equivalently $A \lstella B$
  in $L(\H, \K_{W_2})$) if and only if $P$ admits a bounded extension
  $\tilde{P}\in \Q$ (or equivalently $A\minus B$).

  Here is a proof of the last statement: suppose that there exists a
  weight $W_2$ on $\K$ positive and invertible such that $P$ is
  symmetric with respect to $\ip{\cdot}{\cdot}_{W_2}$.  Since $P$ is a
  (densely defined) idempotent then $\D(P)=R(P)\dotplus N(P)$, where
  $\D(P)$ is the domain of $P$.  Moreover, given $x\in R(P)$ and $y\in
  N(P)$ we have $ \ip{x }{y }_{W_2}= \ip{Px }{y }_{W_2}= \ip{x }{Py
  }_{W_2}=0$ because $P$ is symmetric with respect to
  $\ip{\cdot}{\cdot}_{W_2}$ and $y\in N(P)$.  Hence
  $\D(P)=R(P)\Oplus_{W_2} N(P)$, and consequently
  $\H=\roverline{R(P)}\Oplus_{W_2}\roverline{N(P)}$, where the
  closures are taken with respect to $\ip{\cdot}{\cdot}_{W_2}$.  Then
  $\roverline{P}=P_{\hs\roverline{R(P)}//\roverline{N(P)}}$ is a
  bounded extension of $P$.

  Conversely, suppose that there exists $\tilde{P}\in \Q$ such that
  $P\subseteq \tilde{P}$, and let
  $W_2=\tilde{P}^*\tilde{P}+(I-\tilde{P})^*(I-\tilde{P})$, which is
  positive and invertible and satisfies $W_2\tilde{P}=\tilde{P}^*W_2$.
  Finally, $P$ is symmetric with respect to $\ip{\cdot}{\cdot}_{W_2}$.
  In fact, if $x, y\in \D(P)$ then $ \ip{Px
  }{y}_{W_2}=\ip{\tilde{P}x}{y}_{W_2} = \ip{W_2\tilde{P}x}{y} =
  \ip{x}{W_2\tilde{P}y} = \ip{W_2x}{\tilde{P}y} =
  \ip{x}{\tilde{P}y}_{W_2} = \ip{x}{Py}_{W_2}$.
\end{remark}
\medskip

\begin{corollary}
  \label{djv7l1}
  Let $A,B\in L(\H)$ be such that $A \lminus B$.  Then
  $\cl{R(B^*)}=\cl{R(A^*)}\dotplus \cl{R(B^*-A^*)}$.
\end{corollary}

\begin{proof}
  From Proposition~\ref{leftminus and projections}(3), if $A\lminus
  B$, then $A=PB$ and $N(B)\subseteq N(A)$ or
  $\roverline{R(A^*)}\subseteq \roverline{R(B^*)}$ and in the same
  way, $\roverline{R(B^*-A^*)}\subseteq \roverline{R(B^*)}$.  Then by
  Proposition~\ref{equiv range sum, kernels}(1),
  $\roverline{R(A^*)}\dotplus \roverline{R(B^*-A^*)}\subseteq
  \roverline{R(B^*)}$.  On the other hand, $R(B^*)\subseteq
  R(A^*)+R(B^*-A^*)\subseteq \roverline{R(A^*)}\dotplus
  \roverline{R(B^*-A^*)}$.  Hence
  \begin{equation*}
    R(B^*)\subseteq \roverline{R(A^*)}\dotplus
    \roverline{R(B^*-A^*)}\subseteq \roverline{R(B^*)}.
  \end{equation*}
  But $\roverline{R(A^*)}\dotplus \roverline{R(B^*-A^*)}$ is closed by
  Proposition~\ref{equiv range sum, kernels}.  Therefore,
  $\cl{R(B^*)}=\cl{R(A^*)}\dotplus \cl{R(B^*-A^*)}$.
\end{proof}

\begin{corollary}
  \label{leftminus then minus}
  Let $A, B\in L(\H)$ such that $A \lminus B$.  If $R(B)$ is closed
  then $R(A)$ and $R(B-A)$ are closed and $A\minus B$.
\end{corollary}

\begin{proof}
  Since $A \lminus B$, by Corollary~\ref{djv7l1},
  $\cl{R(B^*)}=\cl{R(A^*)}\dotplus \cl{R(B^*-A^*)}$.  If $R(B)$ is
  closed, then $R(B^*)$ is closed and
  \begin{equation*}
    \roverline{R(A^*)}\dotplus
    \roverline{R(B^*-A^*)}=\roverline{R(B^*)}=R(B^*)\subseteq
    R(A^*)\dotplus R(B^*-A^*).
  \end{equation*}
  Therefore $\roverline{R(A^*)}\dotplus \roverline{R(B^*-A^*)}=
  R(A^*)\dotplus R(B^*-A^*).$ This implies that
  $\roverline{R(A^*)}=R(A^*)$ and $\roverline{R(B^*-A^*)}=R(B^*-A^*)$
  and $A\minus B$.
\end{proof}

The above corollary shows that, unlike the left (right) star order,
the left (right) minus order coincides with the minus order when
applied to matrices.  However for operators these orders are not the
same.

\begin{example}[See also \cite{BakSemSty96}]
  \label{djv7ex1}
  Let $A\in L(\H)$ be an operator such that $R(A)\not = \cl{R(A)}$ and
  that there exists $x\in \cl{R(A)}\setminus R(A)$ which is not
  orthogonal to $N(A)$.  For example, consider $\H=l^2(\mathbb{N})$
  the space of all square-summable sequences, operator $A$ defined as
  $A: (x_n)_{n\in \mathbb{N}} \mapsto ((1/n)x_{n+1})_{n\in
    \mathbb{N}}$, and take $x$ to be $x=(1/n)_{n\in \mathbb{N}}$.
  Define operator $B$ as $B=A+P_x$, where $P_x$ is the orthogonal
  projection onto the one-dimensional subspace spanned by $\{x\}$.
  Since $N(A)\not \subseteq N(P_x)$, and $N(P_x)$ is of co-dimension
  one, we have $\H=N(A)+N(P_x)$, which according to
  Proposition~\ref{equiv range sum, kernels} shows that $A$ and $P_x$
  are range-additive; that is, $R(A)+R(B-A)=R(B)$.  We also have
  $R(A)\cap R(B-A)=\{0\}$ showing that $A\lminus B$.  On the other
  hand, $\cl{R(A)}\cap \cl{R(B-A)}\not = \{0\}$ so $A\minus B$ does
  not hold.
\end{example}

Applying Theorem~\ref{equiv minus and RAP} it is possible to define
the minus order in terms of the inner generalized inverses of the
operators involved.  By an inner inverse of an operator $A\in
L(\H,\K)$ we mean a densely defined operator $A^{-}: \mathcal
D(A^-)\subseteq \K \to \H$ satisfying $R(A)\subseteq \mathcal D(A^-)$
and $AA^{-} A= A$.

\begin{proposition}
  For $A,B\in L(\H)$, the following conditions are equivalent:
  \begin{enumerate}
  \item $A \lminus B$;
  \item there exists an inner inverse $A^-$ of $A$ such that
    $A^-A=A^-B$ and $AA^-x=BA^-x$ for every $x\in \D(A^-)$.
  \end{enumerate}
\end{proposition}

\begin{proof}
  Suppose that $A \lminus B$.  If $\mathcal N$ is a complement of
  $\roverline{R(B)}$, then $\H=\roverline{ R(A)\dotplus R(B-A)}
  \dotplus \mathcal N$.  From Proposition~\ref{equiv range sum,
    kernels} we know that $N(A) + N(B-A)=\H$; so if $\M =
  N(B-A)\ominus N(A)$, then $\H=N(A)\dotplus \M$.  Let $A_1$ be the
  restriction of $A$ to $\M$, and define $A^{-}$ as $A_1^{-1}$ on
  $R(A)$, and as the null operator on $R(B-A)\dotplus \mathcal N$.
  Then $A^-$ is densely defined and the domain of $A^-$ is
  $\D(A^-)=R(B)\dotplus \N$.  In this case, $(A-B)A^{-}x=0$ for every
  $x\in \D(A^-)$, because $R(A^-)\subseteq N(A-B)$.  On the other
  hand, since $R(A-B)\subseteq \D(A^-)$, we find that $A^{-}(A-B)=0$
  since $R(A-B)\subseteq N(A^{-})$.

  For the converse, suppose that there exists an inner inverse $A^-$
  of $A$ such that $A^-A=A^-B$ and $AA^-x=BA^-x$ for every $x\in
  \D(A^-)$.  In particular, if $z\in \H$ then $Az\in R(A)\subseteq
  D(A^-)$, so that $Az=AA^-Az= BA^-Az$.  Hence $R(A)\subseteq R(B)$,
  showing that $R(B)=R(A)+R(B-A)$.  From $A^- A=A^- B$ we have
  $R(A-B)\subseteq N(A^-)$, while $N(A^-)\cap R(A)=\{0\}$, and so
  $R(B)=R(A)\dotplus R(B-A)$.  Therefore, $A\lminus B$.
\end{proof}

\begin{corollary}
For $A,B\in L(\H)$, the following conditions are equivalent:
\begin{enumerate}
  \item $A\minus B$;
  \item there exist inner inverses $A^-$ of $A$ and $(A^*)^-$ of $A^*$
    such that
    \begin{itemize}
    \item[$($i\,$)$] $A^-A=A^-B$ and $AA^-x=BA^-x$ for every $x\in
      \D(A^-)$,
    \item[$($ii\,$)$] $(A^*)^-A^*=(A^*)^-B^*$ and
      $A^*(A^*)^-x=B^*(A^*)^-x$ for every $x\in \D((A^*)^-)$.
    \end{itemize}
  \end{enumerate}
\end{corollary}

\section{Applications}

\subsection{Generalized inverses of $A+B$}

In this section we state the formulas for arbitrary reflexive inverses
of $A+B$ in terms of the inverses of $A$ and $B$, when $A\lminus A+
B$.  For the sake of simplicity, we begin by giving the formula for
the Moore-Penrose inverse.  Theorem~\ref{djt1} states the result in
the most general form, and from this theorem many existing results in
the subject can be recovered.

If $A\minus A+B$ then $A=P(A+B)$ for some $P\in \Q$.  Using the
projection $P$ we can construct a projection $E\in \Q$ onto
$\roverline{R(A+B)}$ that will be useful in stating the formula for
the Moore-Penrose inverse of $A+B$.

\begin{lemma}
  \label{projection E}
  Let $A,B\in L(\H)$ be such that $A\minus A+B$, and $P\in \Q$ be such
  that $A=P(A+B)$.  Set
  \begin{equation*}
    E = P_A P + P_B (I-P).
  \end{equation*}
  Then $E\in \Q$ and $R(E)=\roverline{R(A+B)}$.  Moreover, $E$ is
  selfadjoint if and only if $ P=P_{\M//\N}$ where $\M=
  \cl{R(A)}\Oplus \M_1 $, $\N= \cl{R(B)}\Oplus \N_1 $ with $\M_1$ and
  $\N_1$ closed subspaces such that $\M_1, \N_1\subseteq N(A^*)\cap
  N(B^*)$.
\end{lemma}

\begin{proof}
  If $A=P(A+B)$ then $\roverline{R(A)}\subseteq R(P)$ and
  $\roverline{R(B)}\subseteq N(P)$.  Therefore $ P_A P$ and $P_B
  (I-P)$ are projections, with $R( P_A P)=\roverline{R(A)}$ and $R(P_B
  (I-P))=\roverline{R(B)}$.  Moreover,
  \begin{equation*}
    P_A PP_B (I-P)=P_B (I-P) P_A P=0.
  \end{equation*}
  Therefore $E= P_A P + P_B (I-P)$ is a projection.  Also,
  $\roverline{R(A)}=R( P_A P)=R(EP)\subseteq R(E)$.  Applying
  Lemma~\ref{char1 RAP},
  $R(E)=\roverline{R(A)}+\roverline{R(B)}=\roverline{R(A+B)}$ because
  $A\minus A+B$.

  Finally, if $P=P_{\M//\N}$ then there exist closed subspaces
  $\M_1,\N_1$ such that $\M=\cl{R(A)}\Oplus\M_1$ and
  $\N=\cl{R(B)}\Oplus \N_1$.  Hence
  $P_AP=P_{\hs\cl{R(A)}//{\cl{R(B)}\dotplus \N_1\dotplus \M_1}}$,
  $P_B(I-P)=P_{\hs\cl{R(B)}//{\cl{R(A)}\dotplus \M_1\dotplus \N_1}}$
  and $E=P_{\hs\cl{R(A+B)}//{\M_1\dotplus \N_1}}$.  Since $A\minus
  A+B$, it follows that $E^*=E$ if and only if $\M_1\dotplus
  \N_1=(\cl{R(A+B)})^\perp=(\cl{R(A)}\dotplus
  \cl{R(B)})^\perp=N(A^*)\cap N(B^*)$, or equivalently, $\M_1$ and
  $\N_1$ are included in $N(A^*)\cap N(B^*)$.
\end{proof}

\medskip

\begin{definition}
  \label{def:optimal}
  Let $A, B\in L(\H)$ such that $A\minus A+B$.  Consider $P, Q\in\Q$
  such that $A=P(A+B)=(A+B)Q$, then $P$ will be called \emph{optimal}
  for $A$ and $B$ if $E=P_AP+P_B(I-P)$ is selfadjoint.  In a symmetric
  way, since $A^*=Q^*(A^*+B^*)$, $Q$ will be called \emph{optimal} for
  $A$ and $B$ if $Q^*$ is optimal for $A^*$ and $B^*$, i.e.,
  $F=P_{A^*}Q^*+P_{B^*}(I-Q^*)$ is selfadjoint.
\end{definition}

From the above lemma, the set of optimal projections for $A$ and $B$
is the set
\begin{equation*}
  \left\{P\in \Q: P=P_{\hs\cl{R(A)}\Oplus \M_1//\cl{R(B)}\Oplus
  \N_1} \textrm{ with } \M_1\dotplus \N_1=N(A^*)\cap N(B^*) \right\}.
\end{equation*}

\medskip

Applying Lemma~\ref{projection E} we derive the Fill and Fishkind
\cite{FilFis99} formula for the Moore-Penrose inverse of the sum of
two operators in an easy way.  This formula first appeared in their
work for square matrices, while Gro\ss{} \cite{Gro99} extended it to
arbitrary rectangular matrices and Arias et al.~\cite{AriCorMae15}
proved it for operators on a Hilbert space.  The version we give here
requires simpler hypotheses, and the formula is given in a more
general form.  We include a short proof.

\begin{corollary}
  \label{FFF}
  Let $A,B\in L(\H)$ be such that $R(A+B)$ is closed and $A \lminus A+
  B$.  Then
  \begin{equation}
    \label{djeq5}
    (A+B)^{\dag} = QA^{\dag} P + (I-Q)B^{\dag}(I-P),
  \end{equation}
  where $A=P(A+B)=(A+B)Q$, with $P, Q$ optimal projections for $A$ and
  $B$.
\end{corollary}

\begin{proof} From Corollary~\ref{leftminus then minus} we see that
  $A$ and $B$ are operators with closed range and $A\minus A+B$.  Then
  the operators $A^\dagger$ and $B^\dagger$ are bounded and
  $T=QA^\dagger P+(I-Q)B^\dagger (I-P)$ is well defined.

  Using that $(A+B)Q=A$ and $(A+B)(I-Q)=B$ we have that
  \begin{equation*}
    (A+B)T = AA^\dagger P+BB^\dagger (I-P)=P_A P+P_B(I-P)=P_{R(A+B)},
  \end{equation*}
  by Lemma~\ref{projection E}, because $P$ is optimal.  Using that
  $P(A+B)=A$, $(I-P)(A+B)=B$ we see that
  \begin{equation*}
    T(A+B) = QA^\dagger A + (I-Q)B^\dagger B=QP_{A_\rs^*}
    +(I-Q)P_{B_\rs^*}=P_{R(A_\rs^*+B_\rs^*)}=P_{N(A+B)^\perp}.
  \end{equation*}
  Therefore $T=(A+B)^\dagger$.
 \end{proof}

\begin{remark}
  In \cite[Theorem~5.2]{AriCorMae15} the Fill-Fishkind formula is
  stated as follows: let $A, B\in L(\H)$ be such that $R(A)$, $R(B)$
  are closed, $R(A+B)=R(A)\dotplus R(B)$ and
  $R(A^*+B^*)=R(A^*)\dotplus R(B^*)$, then
  \begin{equation*}
    (A+B)^{\dag} = (I-S)A^\dag(I - T) + S B^\dag T,
  \end{equation*}
  where $S=(P_{N(B)_\rs^\bot}P_{N(A)})^\dag$ and $T=(P_{N(A_\rs^*)}
  P_{N(B_\rs^*)_\rs^\bot})^\dag$.  It holds that $S, T\in \Q$ (see
  \cite[Lemma~2.3]{Pen55} and \cite[Theorem~1]{Gre74} for matrices and
  \cite[Theorem~4.1]{CorMae10} for operators in Hilbert spaces).  If
  we denote $Q=I-S$ and $P = I - T$, we in fact have (see
  \cite[Theorem~5.1]{AriCorMae15}) $Q=P_{R(A_\rs^*)\Oplus (N(A)\cap
    N(B)) // R(B_\rs^*)}^*$ and $P=P_{R(A)\Oplus(N(A_\rs^*)\cap
    N(B_\rs^*)) // R(B)}$, which are optimal with respect to $A\minus
  A+B$.
\end{remark}

Recall that if $A\in \L(\H,\K)$ is a closed range operator, then any
operator $X\in \L(\K,\H)$ satisfying $AXA=A$ and $XAX=X$ is called a
\emph{reflexive inverse} of $A$.  The operator $X$ has closed range,
$AX=P_{R(A)//N(X)}$ and $XA=P_{R(X)//N(A)}$.  If $\M$ and $\N$ are
arbitrary closed subspaces of $\H$ and $\K$ satisfying $R(A)\dotplus
\M=\K$ and $\N\dotplus N(A)=\H$, then there is only one reflexive
inverse of $A$ with the range $\N$ and the null-space $\M$.  This
reflexive inverse is denoted by $A^{(1,2)}_{\N,\M}$.

In what follows we generalize Lemma~\ref{projection E} in order to
prove a formula similar to \eqref{djeq5} in Corollary~\ref{FFF} for an
arbitrary reflexive inner inverse of the sum of two operators.

\begin{lemma}
  \label{djl1}
  Let $A,B\in L(\H)$ be such that $A\minus A+B$.  Let $P\in \Q$ such
  that $A=P(A+B)$ and consider
  \begin{equation*}
    E  =  P_{\hs\cl{R(A)}//\N_1} P  + P_{\hs\cl{R(B)}//\N_2} (I-P),
  \end{equation*}
  where $\N_1$ and $\N_2$ are arbitrary.  Then $E\in \Q$ and
  $R(E)=\roverline{R(A+B)}$.  Moreover, for every closed subspace $\M$
  such that $\cl{R(A+B)}\dotplus \M = \H$ there exist $P\in \Q$ and
  subspaces $\N_1$ and $\N_2$, such that $A=P(A+B)$ and $E=
  P_{\hs\cl{R(A)}//\N_1} P + P_{\hs\cl{R(B)}//\N_2}
  (I-P)=P_{\hs\cl{R(A+B)}//\M}$.
\end{lemma}

\begin{proof}
  In the same way as in the proof of Lemma~\ref{projection E}, it can
  be proved that $E\in \Q$ and $R(E)=\cl{R(A+B)}$.

  To prove the last assertion, take
  \begin{equation}
    \label{djeq1}
    \N_1 = \cl{R(B)}\dotplus \M, \quad \N_2 = \cl{R(A)}\dotplus
    \M,
  \end{equation} 
  and $P=P_{\hs\cl{R(A)}//\cl{R(B)}\dotplus \M}$.  Then $\N_1$ and
  $\N_2$ are closed, the projections $P_{\hs\cl{R(A)}//\N_1}=P$ and
  $P_{\hs\cl{R(B)}//\N_2}$ are well defined and $E=
  P_{\hs\cl{R(A)}//\N_1} P + P_{\hs\cl{R(B)}//\N_2} (I-P)= P +
  P_{\hs\cl{R(B)}//\N_2} (I-P)=P_{\hs\cl{R(A+B)}//\M}$.
\end{proof}

\medskip

\begin{definition}
  \label{def:agrees}
  Let $A, B\in L(\H)$ be such that $A\minus A+B$ and $P\in \Q$ such
  that $A=P(A+B)$.  Given $\M$ an arbitrary closed subspace such that
  $\cl{R(A+B)}\dotplus \M=\H$, we say that $P$ \emph{agrees} with $\M$
  if there exist subspaces $ \N_1,\N_2$ so that
  \begin{equation}\label{djeq6}
    P_{\hs\cl{R(A)}//\N_1} P  + P_{\hs\cl{R(B)}//\N_2} (I-P)=
    P_{\hs\cl{R(A+B)}//\M}.
  \end{equation}

  In a symmetric way, if $A=(A+B)Q$, for $Q \in \Q$ and $\N\dotplus
  N(A+B)=\H$, we say that $Q$ \textit{agrees} with $\N$ if $Q^*$
  agrees with $\N^\perp$.  In this case $A^*=Q^*(A^*+B^*)$ and
  $\roverline{R(A^*+B^*)}\dotplus \N^\perp=\H$, and there exist closed
  subspaces $\N_1^*, \N_2^*$ such that
  \begin{equation}
    \label{djeq7}
    P_{\hs\cl{R(A^*)}//(\N_1^*)_\rs^\perp} Q^* +
    P_{\hs\cl{R(B^*)}//(\N_2^*)_\rs^\bot}(I-Q^*)=P_{\hs\cl{R(A^*+B^*)} //
      \N^\bot},
  \end{equation}
  or
  \begin{equation}
    \label{djeq7'}
    QP_{\N_1^*//N(A)}  + (I-Q)P_{\N_2^*//N(B)}=P_{\N // N(A+B)}.
  \end{equation}
\end{definition}

\medskip

For example, $P=P_{\hs\cl{R(A)}//\cl{R(B)}\dotplus \M}$ agrees with
$\M$, as we saw in the proof of Lemma~\ref{djl1}, and
$Q=(P_{\hs\cl{R(A^*)}//\cl{R(B^*)}\dotplus \N^\bot})^* = P_{N(B)\cap \N
  // N(A)}$ agrees with $\N$.  The projection $P$ is optimal if $P$
agrees with $R(A+B)^\perp$.

\medskip

\begin{theorem}
  \label{djt1}
  Let $A,B\in L(\H)$ be such that $R(A+B)$ is closed and $A\lminus
  A+B$.  Let $\M$ and $\N$ be two closed subspaces such that
  $R(A+B)\dotplus \M=\H$ and $\N\dotplus N(A+B)=\H$.  If $P, Q\in \Q$
  satisfy $A=P(A+B)=(A+B)Q$ and agree with $\M$ and $\N$ respectively,
  then
  \begin{equation}
    \label{djeq9}
    (A+B)^{(1,2)}_{\N,\M} = Q A^{(1,2)}_{\N_1^*,\N_1} P +
    (I-Q)B^{(1,2)}_{\N_2^*,\N_2} (I-P),
  \end{equation}
  where the subspaces $\N_1$, $\N_2$, $\N_1^*$ and $\N_2^*$ are
  arbitrary closed subspaces satisfying \eqref{djeq6} and
  \eqref{djeq7'}.
\end{theorem}

\begin{proof} From Corollary~\ref{leftminus then minus} we see that
  $A$ and $B$ are operators with closed range and $A\minus A+B$.  Let
  $T = Q A^{(1,2)}_{\N_1^*,\N_1} P + (I-Q)B^{(1,2)}_{\N_2^*,\N_2}
  (I-P)$ and choose $E$ as in Lemma~\ref{djl1}.  Using that $(A+B)Q=A$
  and $(A+B)(I-Q)=B$ and Lemma~\ref{djl1},
  \begin{equation*}
    (A+B)T = E= P_{R(A+B)//\M}.
  \end{equation*}
  In a similar way,  since $P(A+B)=A$ and $(I-P)(A+B)=B$,
  \begin{equation*}
    T(A+B) =P_{\N//N(A+B)}.
  \end{equation*}
  This shows that $T$ is the reflexive inverse of $A+B$, associated to
  $\M$ and $\N$.
\end{proof}

In fact, the terms on the right hand side of \eqref{djeq9} do not
depend on the choices of the subspaces $\N_1, \N_1^*, \N_2$ and
$\N_2^*$.

\begin{proposition}
  \label{djp1}
  Under the hypotheses of Theorem~\ref{djt1} it holds
  \begin{equation}
    \label{djeq8}
    Q A^{(1,2)}_{\N_1^*,\N_1} P = A^{(1,2)}_{N(B)\cap \N, R(B)\dotplus
      \M} \quad \mbox{ and } \quad (I-Q)B^{(1,2)}_{\N_2^*,\N_2} (I-P)
    = B^{(1,2)}_{N(A)\cap \N, R(A)\dotplus \M}
  \end{equation}
  regardless of the choice of $P$, $Q$, $\N_1$, $\N_2$, $\N_1^*$ and
  $\N_2^*$.  Consequently
  \begin{equation}
    \label{djeq10}
    (A+B)^{(1,2)}_{\N,\M} =  A^{(1,2)}_{N(B)\cap \N, R(B) \dotplus \M}
    + B^{(1,2)}_{N(A)\cap \N, R(A) \dotplus \M}. 
  \end{equation}
\end{proposition}

\begin{proof}
  If $T=Q A^{(1,2)}_{\N_1^*,\N_1} P$, then since $A=P(A+B)=(A+B)Q$, we
  get $A=PA=AQ$.  Thus $ATA=A$ and $TAT=T$, so $T$ is a reflexive
  inverse of $A$.  Besides that, $AT = P_{R(A)//\N_1} P$ and so
  $N(AT)=N(P)\dotplus (R(P) \cap \N_1)$.  From Lemma~\ref{djl1} we
  have that $R(P)=R(A)\dotplus \M_1$, $N(P)=R(B)\dotplus \M_2$, where
  $\M_1\dotplus \M_2=\M$.  Hence $N(P)\dotplus (R(P) \cap \N_1) =
  R(B)\dotplus \M$.  This shows that $N(T)=R(B)\dotplus \M$.  On the
  other hand, $R(TA)=R(QP_{\N_1^*//N(A)})=N(P_{R(A_\rs^*)//(\N_1^*)^\bot
  }Q_\rs^*)^\bot$.  By an identical argument with $A^*$ and $B^*$ in place
  of $A$ and $B$ we have $N(P_{R(A_\rs^*)//(\N_1^*)^\bot}
  Q^*) = R(B^*)\dotplus \N^\bot$.  Hence $R(TA)=N(B)\cap \N$.  This
  shows that $T= A^{(1,2)}_{N(B)\cap \N, R(B) \dotplus \M}$.  The
  proof for $B$ follows along similar lines.
\end{proof}

\medskip

\begin{remark}
  \label{djr1}
  Formula \eqref{djeq10} was given for matrices by Werner in
  \cite{Wer86}.  There the author considers pairs of matrices $A$
  and $B$ having the property $R(A)\cap R(B)=R(A^*)\cap R(B^*)=\{0\}$,
  and calls such matrices \textit{weakly bicomplementary}.  Recall
  that in the finite-dimensional setting, conditions $R(A)\dotplus
  R(B)=R(A+B)$, $R(A^*)\dotplus R(B^*)=R(A^*+B^*)$ and $R(A)\cap
  R(B)=\{0\}=R(A^*)\cap R(B^*)$ are all equivalent.  Also matrices $A$
  and $B$ are weakly bicomplementary if and only if $A\minus A+ B$.
  Under this assumption, many of the results from \cite{Wer86} are
  seen to hold in arbitrary Hilbert spaces.
\end{remark}

\medskip

Recall that if for $A\in L(\H)$ the relation $\H=R(A)\dotplus N(A)$
holds, then $R(A)$ is closed, see \cite[Proposition~3.7]{Dix49} or
\cite[Theorem~2.3]{FilWil71}.  In this case, $A$ is called
\textit{group-invertible} and $A^{\ssharp} = A^{(1,2)}_{R(A),N(A)}\in
L(\H)$ is called the \textit{group inverse} of $A$.

If $A$ is group invertible, the operator $A^{\core} =
A^{(1,2)}_{R(A),N(A_\rs^*)}\in L(\H)$ is called the \textit{core inverse}
of $A$ and it was introduced by Baksalary and Trenkler
\cite{BakTre10}, see also \cite{JosSiv15} and
\cite{RakDinDjo14}.

If we denote by $L^{1}(\H)$ the set of all
group invertible operators, then the \textit{sharp partial order} and
the \textit{core partial order} on $L^1(\H)$ are defined as:
$A\sharppo B$ if $AA^{\ssharp} = BA^{\ssharp}$ and $A^{\ssharp}
A=A^{\ssharp} B$; $A\corepo B$ if $AA^{\core} = BA^{\core}$ and
$A^{\core} A=A^{\core} B$.  It is straightforward to see that for
$A,B\in L^1(\H)$ we have
\begin{equation}
  \label{djsh}
  A\sharppo B \quad \Leftrightarrow \quad A^2 =BA = AB,
\end{equation}
and
\begin{equation}
  \label{djco}
  A\leq^{\core} B \quad \Leftrightarrow \quad A^*A=A^*B \quad
  \mbox{and}\quad A^2=BA.
\end{equation}
We recover results from Jose and Sivakumar \cite{JosSiv15} in the
following corollary.

\begin{corollary}
  \label{djc2} 
  Let $A,B\in L(\H)$ such that $R(A+B)$ is closed.
  \begin{enumerate}
  \item If $A \stella A+B$ then $(A+B)^{\dag} = A^{\dag} + B^{\dag}$;
  \item If $A,A+B\in L^1(\H)$ and $A\sharppo A+B$, then $B\in
    L^1(\H)$ and $(A+B)^{\ssharp} = A^{\ssharp} + B^{\ssharp}$;
  \item If $A,A+B\in L^1(\H)$ and $A\corepo A+B$ then $B\in L^1(\H)$.
    If moreover $A^*\corepo A^* + B^*$ then $(A+B)^{\core} = A^{\core}
    + B^{\core}$.
  \end{enumerate}
\end{corollary}

\begin{proof}
  All three partial orders stated here induce the minus partial order,
  so $R(A)$ and $R(B)$ are closed, according to Corollary~\ref{A minus
    B and rg B cl then rg A, rg B-A cl}.

  To begin with, $(1)$ follows from Proposition~\ref{djp1}.

  For $(2)$, by \cite[Corollary~3.5]{JosSiv15} we see that $B\in
  L^1(\H)$ and $B\sharppo A+B$.  It is also easy to see that
  $N(B)=R(A)\dotplus (N(A)\cap N(B))$ and $N(A)=R(B)\dotplus (N(A)\cap
  N(B))$.  Recall that $N(A+B)=N(A)\cap N(B)$, since $A\minus A+B$.
  Now apply Proposition~\ref{djp1} to obtain the desired relation.

  Finally for (3), we have from \cite[Theorem~4.5]{JosSiv15} that
  $B\in L^{1}(\H)$ and moreover, from \eqref{djsh} and \eqref{djco} we
  see that $A\stella A+B$ and $A\sharppo A+B$.  Now taking
  $\M=N(A^*)\cap N(B^*)$ and $\N=R(A+B)$ and combine $(1)$ with $(2)$
  with Proposition~\ref{djp1} to get the desired result.
\end{proof}

\medskip

\subsection{Systems of equations and least squares problems}

Consider $A,B\in L(\H)$.  In what follows we characterize the left
minus order in terms of the solutions of the equation $(A+B)x=c$, for
$c\in R(A+B)$.

The following theorem appears in \cite{Wer86} in the matrix case.
We give the proof for operator equations.

\begin{proposition}
  If $A,B\in L(\H)$ the following statements are equivalent:
  \begin{enumerate}
  \item  $A \lminus A+B$;
  \item Given $a\in R(A)$ and $b\in R(B)$, the equation 
    \begin{equation}
      \label{eq A+B}
      (A+B)x=a+b,
    \end{equation}
    has a solution $x_0\in\H$.  Moreover, $x_0$ is a solution of the
    system
    \begin{equation}
      \label{eq A y B}
      \begin{cases}
        Ax=a&\\
        Bx =b.&
      \end{cases}
    \end{equation}
  \end{enumerate}
\end{proposition}

\begin{proof}
  Suppose that $A \lminus A+B$ and $a\in R(A)$ and $b\in R(B)$.  Since
  $A \lminus A+B$, then $R(A+B)=R(A)\dotplus R(B)$.  Therefore,
  $a+b\in R(A+B)$, so that there exists $x_0\in \H$ satisfying
  $(A+B)x_0=a+b$.  Moreover, $Ax_0-a=b-Bx_0=0$ since $R(A)\cap
  R(B)=\{0\}$.  Hence $x_0$ is a solution of the system \eqref{eq A y
    B}.

  For the converse, let $a\in R(A)$.  By hypothesis, there exists
  $x_0\in \H$ such that $(A+B)x_0=a$.  Hence $R(A)\subseteq R(A+B)$,
  so that $R(A+B)=R(A)+R(B)$.  Now, let $c\in R(A)\cap R(B)$.
  Consider $x_0\in \H$ such that $(A+B)x_0=c$.  Since $c\in R(A)\cap
  R(B)$ then, also by hypothesis,
  \begin{equation*}
    \begin{cases}
      Ax_0=c & \\
      Bx_0 =0 &
    \end{cases}
    \quad\textrm{ and  }\quad 
    \begin{cases}
      Ax_0=0 & \\
      Bx_0 =c.  &
    \end{cases}
  \end{equation*}
  Therefore $c=0$, and so $R(A+B)=R(A)\dotplus R(B)$, or equivalently,
  $A\lminus A+B$.
\end{proof}

\medskip

More generally, in what follows we relate the least squares solutions
of the equation $Cx=y$ to a weighted least squares solution of the
system $Ax=y$ and $(C-A)x=y$ when $A \lminus C$.  We introduce the
seminorm given by a positive weight.  If $W\in L(\H)$ is a positive
(semidefinite), consider the seminorm $\|.\|_W$ on $\H$ defined by
\begin{equation*}
  \| x\|_W= \ip{W x}{ x}^{1/2}, \quad x\in \H.
\end{equation*}
Given $C\in L(\H)$ and $y\in \H$, an element $x_0\in \H$ is said a
$W$-\textit{least squares solution} ($W$-LSS) of the equation $Cx=y$
if
\begin{equation*}
  \|Cx_0-y \|_W= \underset{x\in \H}{\textrm{min }}  \|Cx-y\|_W.
\end{equation*}
It is well known that $x_0\in \H$ is a $W$-LSS of the equation $Cx=y$
if and only if $x_0$ is a solution of the associated \textit{normal
  equation},
\begin{equation*}
  C^*W(Cx-y)=0.
\end{equation*}

\begin{proposition}
  \label{LSS and WLSS}
  Let $A, B\in L(\H)$ be such that $R(A+B)$ is closed and $A \lminus
  A+B$ and $c\in \H$.  For $P\in \Q$ is any optimal projection such
  that $A=P(A+B)$, let $W=P^*P+(I-P^*)(I-P)$.  Then the following
  statements are equivalent:
  \begin{enumerate}
  \item $x_0$ is a solution of 
    \begin{equation}\label{LSP}
      \underset{x\in \H}{\mathrm{argmin\,}}\| (A+B)x-c\|;
    \end{equation}
  \item $x_0$ is a solution of the system of least squares problems,
    \begin{equation}\label{SLSP}
      \left\{\begin{array}{l}
          \underset{x\in \H}{\mathrm{argmin\,}}\|Ax-c\|_{W}\\
          \underset{x\in \H}{\mathrm{argmin\,}}\| Bx-c\|_{W}. 
        \end{array}
      \right.
    \end{equation}
  \end{enumerate}
\end{proposition}

\begin{proof}
  Assume $(1)$ holds.  Applying Corollary~\ref{leftminus then minus},
  if $A \lminus A+B$ and $R(A+B)$ is closed then $R(A)$ and $R(B)$ are
  closed and $A\minus A+B$.  Suppose that $x_0$ is a solution of
  \eqref{LSP} then $x_0$ is a solution of the associated normal
  equation
  \begin{equation*}
    (A+B)^*((A+B)x-c)=0.
  \end{equation*}
  Then 
  $(A+B)x_0-c\in N(A^*+B^*)=N(A^*)\cap N(B^*)$ so that
  \begin{equation*}
    E((A+B)x_0-c)=0,
  \end{equation*}
  where $E=P_AP+P_B(I-P)$ is the orthogonal projection onto $R(A+B)$
  because $P$ is optimal.  Therefore,
  \begin{equation*}
    Ax_0-P_APc=-(Bx_0-P_B(I-P)c),
  \end{equation*}
  and because $R(A)\cap R(B)=\{0\}$, we have
  $Ax_0-P_APc=Bx_0-P_B(I-P)c=0$.  Hence $P_A(Ax_0-Pc)=0$ and
  $P_B(Bx_0-(I-P)c)=0$, or equivalently,
  \begin{equation*}
    A^*(Ax_0-Pc)=0\; \textrm{ and }\; B^*(Bx_0-(I-P)c)=0.
  \end{equation*}
  Thus $A^*P(Ax_0-c)=0$ and $B^*(I-P)(Bx_0-c)=0$.

  Finally, observe that $A^*P=A^*P^*P=A^*W$, where
  $W=P^*P+(I-P^*)(I-P)$.  Then $x_0$ is a solution of the normal
  equation
  \begin{equation*}
    A^*W(Ax-c)=0.
  \end{equation*}
  Equivalently, $x_0$ is a solution of
  \begin{equation*}
    \underset{x\in \H}{\mathrm{argmin\,}}\|Ax-c\|_{W}.
  \end{equation*}
  In the same way, the equation $B^*(I-P)(Bx-c)=0$ is equivalent to
  $B^*W(Bx-c)=0$ which is the normal equation of the minimizing
  problem
  \begin{equation*}
    \underset{x\in \H}{\mathrm{argmin\,}}\|Bx-c\|_{W}.
  \end{equation*}

  Next consider the converse.  As we noted above, $x_0$ is a solution
  of the system \eqref{SLSP} if and only if $x_0$ is a solution of
  \begin{equation*}
    \begin{cases}
      A^*(Ax-Pc)=0 &\\
      B^*( Bx-(I-P)c)=0. &
    \end{cases}
  \end{equation*}
  We show in this case, that $A^*(Bx_0-(I-P)c)=0$.  In fact, from the
  second equation we have that $Bx_0-(I-P)c \in R(B)^\perp$, since
  applying $P_{R(B)^\perp}$ to $Bx_0-(I-P)c$, we get that
  $Bx_0-(I-P)c=-P_{R(B)^\perp}(I-P)c$.  If $P$ is optimal, then
  $P=P_{{R(A)}\Oplus \mathcal M_1// R(B)\Oplus \N_1}$ where $\mathcal
  M_1\dotplus \mathcal N_1=N(A^*)\cap N(B^*)$.  Thus, $Bx_0-(I-P)c\in
  P_{R(B)^\perp}(N(P))=\mathcal N_1\subseteq N(A^*)$.  Therefore
  $A^*(Bx_0-(I-P)c)=0$.  In the same way, $B^*(Ax_0-Pc)=0.$ The sum of
  $A^*(Ax_0-Pc)=0$ and $A^*(Bx_0-(I-P)c)=0$ gives $A^*((A+B)x_0-c)=0$.
  Analogously, $B^*((A+B)x_0-c)=0$.  Then $(A^*+B^*)((A+B)x_0-c)=0$,
  and so $x_0$ is a solution of the problem $\underset{x\in
    \H}{\mathrm{argmin\,}}\|(A+B)x-c\|_{W}$.
\end{proof}

\medskip

\begin{corollary}
  Let $A, B\in L(\H)$ be such that $R(A+B)$ is closed and $A \lstella
  A+B$ and $c\in \H$.  Then the following statements are equivalent:
  \begin{enumerate}
  \item $x_0$ is a solution of
    \begin{equation*}
      \underset{x\in \H}{\mathrm{argmin\,}}\| (A+B)x-c\|;
    \end{equation*}
  \item $x_0$ is a solution of the following system of least squares
    problems:
    \begin{equation*}
      \begin{cases}
        \underset{x\in \H}{\mathrm{argmin\,}}\|Ax-c\| & \\
        \underset{x\in \H}{\mathrm{argmin\,}}\| Bx-c\|. &
      \end{cases}
    \end{equation*}
\end{enumerate}
\end{corollary}

\begin{proof}
  This follows from Proposition~\ref{LSS and WLSS} and the fact that
  we can take $P=P_A$ as an optimal projection such that $A=P(A+B)$
  (see Proposition~\ref{left star order entonces minus order}).
  In this case, $W=P^*P+(I-P^*)(I-P)=I$.
\end{proof}

\bigskip

\begin{finalremark}
  It is possible to define a weak version of the minus order in the
  following way: Consider $A, B\in L(\H)$, we write $A\minus_w B$ if
  there exist two densely defined idempotent operators $P,Q$ with
  closed ranges such that $A=PB$, $R(P)=\roverline{R(A)}$, $A^*=QB^*$
  and $R(Q)=\roverline{R(A^*)}$.  The relation $\minus_w$ is a partial
  order on $L(\H)$.  In fact, it is not difficult to see that the
  relation is reflexive and antisymmetric.

  For transitivity, consider $A, B, C\in L(\H)$ such that $A \minus_w
  B$ and $B\minus_w C$.  By definition there exist $P_1, P_2, Q_1,
  Q_2$ densely defined idempotent operators such that $A=P_1B$,
  $R(P_1)=\roverline{R(A)}$, $A^*=Q_1B^*$,
  $R(Q_1)=\roverline{R(A^*)}$, $B=P_2C$ and $R(P_2)=\roverline{R(B)}$.
  $B^*=Q_2A^*$ and $R(Q_2)=\roverline{R(B^*)}$.  Observe that
  $A=P_1P_2C$.  Without loss of generality, suppose that
  $P_1=P_{\hs\roverline{R(A)}//{R(B-A)}\Oplus \mathcal M_1}$ and
  $P_2=P_{\hs\roverline{R(B)}//{R(C-B)}\Oplus \mathcal M_2}$, with
  $\mathcal M_1=R(B)^\perp$ and $ \mathcal M_2=R(C)^\perp$.  Let
  $\mathcal D=\roverline{R(A)}\dotplus R(B-A)\dotplus R(C-B)\Oplus
  \mathcal M_2$.  Note that $\mathcal D$ is dense and $R(C)\subseteq
  \mathcal \D\subseteq D(P_1P_2)$, where the second inclusion follows
  because $P_2 x=0$ for all $ x \in R(C-B)\Oplus \mathcal M_2$ and
  $P_2 x =x $ for all $x \in \roverline{R(A)}\dotplus R(B-A)\subseteq
  \roverline{R(B)}\subseteq \D(P_1)$.  Consider $P=P_1P_2|_{\mathcal
    D}$, then $PC=P_1P_2|_{\mathcal D}C=P_1P_2C=A$.  If $x\in
  \roverline{R(A)}$, then $Px=P_1P_2|_{\mathcal D}x=P_1x=x$, because
  $\roverline{R(A)}\subseteq \roverline{R(B)}$.
  Therefore, $\roverline{R(A)}\subseteq R(P)\subseteq
  R(P_1)=\roverline{R(A)}$ so that $R(P)=\roverline{R(A)}$.  Since
  $R(P)=\roverline{R(A)}\subseteq \D(P)$ and
  $P^2=P_1P_2P_1P_2|_{\mathcal D}=P_1P_1P_2|_{\mathcal D}=P$, the
  operator $P$ is idempotent.  Therefore, $P$ is a densely defined
  projection with closed range such that $A=PC$.  Similarly, it
  follows that there exists a densely defined projection $Q$ with
  closed range such $A^*=QC^*$.  Hence $A\minus_w C$, and so
  $\minus_w$ is a partial order.

  Moreover, if $A, B\in L(\H)$ it can be proved that the following
  statements are equivalent:
  \begin{enumerate}
  \item $A\minus_w  B$;
  \item $ \roverline{R(A)}\cap
    {R(B-A)}=\roverline{R(A^*)}\cap{R(B^*-A^*)}=0$;
  \item $ \roverline{R(B)}= \roverline{\roverline{R(A)}\dotplus
      {R(B-A)}}$ and
    $\roverline{R(B^*)}=\roverline{\roverline{R(A^*)}\dotplus
      {R(B^*-A^*)}}$.
  \end{enumerate} 

  Finally, note that $\minus_w$ is weaker than the minus order.  In
  fact, consider $P, Q\in \P$ such that $R(P)\cap R(Q)=0$ and
  $c_0(R(P),R(Q))=1$.  Then by the equivalence $(1) \Leftrightarrow
  (2)$ above, $P\minus_w P+Q$.  However, by
  Proposition~\ref{equivalencias minus order}, it is not the case that
  $P\minus P+Q$.
\end{finalremark}

\end{document}